\documentclass[10pt]{article}

\usepackage{graphicx}
\usepackage{amsmath}
\usepackage{amsfonts}
\usepackage{amssymb}
\usepackage{amsthm}
\usepackage{eucal}
\usepackage{nccmath}
\usepackage{array}
\usepackage{bbm}
\usepackage[margin=1in]{geometry}

\newtheorem{lemma}{Lemma}

\newcommand{\tr}[0]{\,\text{tr}\,}
\newcommand{\vect}[2]{\left( \begin{array}{l} #1 \\ #2 \end{array} \right)}
\newcommand{\matr}[4]{\left( \begin{array}{ll} #1 & #2 \\ #3 & #4 \end{array}\right)}

\newcommand{\intinfty}[0]{\frac{1}{2\pi i} \oint_\infty}
\newcommand{\irho}[0]{\left( {\scriptstyle \int}_{a_-}^\lambda \rho^{(1)}_N \right)}

\begin{document}

\title{Valence independent formula for the equilibrium measure}
\author{Patrick Waters}
\maketitle

\abstract{We derive a new formula for the equilibrium measure for eigenvalues of random matrices sampled from polynomial perturbations of the GUE, valid in the one-cut case.  The virtue of our formula is that it depends on the potential only implicitly through the endpoints of support of the equilibrium measure.  Our motivation is the problem of computing explicit formulas for generating functions which enumerate graphs embedded in a Riemann surface.  To demonstrate the utility of our formula for the equilibrium measure, we derive a formula for the generating function $e_1$ enumerating maps on the torus.  This formula is ``valence independent'' in the sense that it holds regardless of what numbers of edges are allowed to meet at vertices; furthermore it subsumes formulas for $e_1$ given in \cite{BIZ},\cite{EMcLP08}, and \cite{EP11} as special cases.}

\section{Introduction and statement of results}

\subsection{Background on equilibrium measures}
In this article we study the equilibrium measure $\psi(\lambda)\,d\lambda$ associated with a potential function 
\begin{align}
V(\lambda) =& \frac{1}{x} \left( \frac{1}{2} \lambda^2 + \sum_{j=1}^d t_j \lambda^j \right)
\end{align}
depending on parameters $x,t_1,\ldots,t_d$.  This measure is characterized as the unique solution to the following variational problem \cite{DeiftOPs}.

\emph{Find a function $\psi$ supported on a compact subset of $\mathbb{R}$ and a constant $l$ such that}
\begin{align}
g\{\psi\}(\lambda)_+ + g\{\psi\}(\lambda)_- =& l + V(\lambda), & \lambda\in \text{supp}\,(\psi), \label{varEq} \\
g\{\psi\}(\lambda)_+ + g\{\psi\}(\lambda)_- \leq & l + V(\lambda), & \lambda\notin \text{supp}\,(\psi), \label{varIneq}
\end{align}
\emph{where the plus and minus subscripts are boundary values from the upper and lower half planes, and}
\begin{align} \label{gDef}
g\{\psi\}(\lambda) =& \int_\mathbb{R} \log(\lambda -s)\psi(s)\, ds.
\end{align}

The motivation for studying such a measure comes from random matrix theory.
Consider the following probability distribution on random $N\times N$ Hermitian matrices $M$.
\begin{align}
dP_N(M) =& \frac{1}{Z_N}\exp\left(-N \tr V(M)\right)\, dM, \label{measure1} \\
dM=& \prod_{1\leq i<j\leq N} dM^{(Re)}_{ij}dM^{(Im)}_{ij} \prod_{i=1}^N dM_{ii}
\end{align} 
If $t=0$ so that $V(M)=\frac{1}{2}M^2$, then (\ref{measure1}) becomes the Gaussian Unitary Ensemble (GUE).  
One of the cornerstones of random matrix theory is Wigner's semicircle rule, which states that as $N\rightarrow\infty$, the mean density of eigenvalues for $N\times N$ GUE random matrices converges to a (squashed) semicircle.  For an interval $(a,b)$,
\begin{align} \label{Wigner1}
\lim_{N\rightarrow\infty} \frac{1}{N}\mathbb{E}\left[\# \text{ eigenvalues in }(a,b)\right] = \int_{a}^b \frac{1}{2\pi} \mathbbm{1}_{[-2,2]}(\lambda) \sqrt{4-\lambda^2} \, d\lambda.
\end{align}
For the random matrices defined by (\ref{measure1}), the mean density of eigenvalues will converge to the equilibrium measure for the potential $V(\lambda)$ \cite{DeiftOPs}.  Thus Wigner's semicircle rule states that the equilibrium measure corresponding to the potential $V(\lambda)=\frac{1}{2}\lambda^2$ is the semicircle distribution (\ref{Wigner1}).

For the GUE, the equilibrium measure is supported on the interval $[-2,2]$.  We are interested in the case where the equilibrium measure is supported on a single interval.  Ercolani and McLaughlin gave a sufficient condition for this in \cite{EMcL03}.  For $r,\gamma>0$ define
\begin{align}
\mathbb{T}(r,\gamma)= \left\{ t \in \mathbb{R}^d : \; |t|<r,\; t_d> \gamma \sum_{j=1}^{d-1}|t_j|\right\}.
\end{align}
If $d$ is even then there exist $r,\gamma>0$ such that for $t\in \mathbb{T}(t,\gamma)$ and $x$ in some neighborhood of 1, the equilibrium measure is supported on a single interval.  
This is known in the literature as the ``one-cut case''.  For one-cut potentials, there is a well known formula for the equilibrium measure.  With the following notation for Laurent coefficients 
\begin{align}
[y^{r}]_{\infty} F(y) =& \, \text{coefficient of $y^r$ in Laurent expansion of $F(y)$ about $\infty$} \nonumber \\
=&  \lim_{R\rightarrow\infty} \frac{1}{2\pi i} \oint_{\partial B_R(0)} y^{-r} F(y) \, \frac{dy}{y},
\end{align}
the equilibrium measure corresponding to a one-cut potential is given by the following formula \cite{DeiftOPs}.
\begin{align}
\psi(\lambda) =& \frac{1}{2\pi x } \mathbbm{1}_{[\alpha_-,\alpha_+]}(\lambda) \label{psiFormula} \sqrt{(\alpha_+-\lambda)(\lambda-\alpha_-)}\, h(\lambda) \\
h(\lambda)=& [y^{-1}]_{\infty} \frac{xV'(y)}{(y-\lambda)\sqrt{(y-\alpha_-)(y-\alpha_+)}}. \label{oldHFormula}
\end{align} 
Also, the endpoints of support of the equilibrium measure can be determined from the equations
\begin{align}
0 =& [y^{-1}]_{\infty} \frac{ V'(y)}{\sqrt{(y-\alpha_-)(y-\alpha_+)}}  \label{endPts1} \\
2 =& [y^{-2}]_{\infty} \frac{ V'(y)}{\sqrt{(y-\alpha_-)(y-\alpha_+)}}. \label{endPts2}
\end{align}
Notice that the above formula for the equilibrium measure depends on the time parameters $t$ both explicitly in through $V(\lambda)$, and implicitly through $\alpha_\pm$.  It is often natural to make the following change of variables
\begin{align}
u=& \frac{\alpha_+ + \alpha_-}{2} \\
z=& \frac{(\alpha_+ - \alpha_-)^2}{16}.
\end{align}
If the potential $V(\lambda)$ is an even function, then $u=0$.  A heuristic reason for this change of variables is that the functions $u,z$ are the limits as $N\rightarrow\infty$ of sequences of recurrence coefficients for a family of orthogonal polynomials associated with the matrix ensemble \cite{EMcLP08}.

In this article we derive a formula for the equilibrium measure that depends on $t$ \emph{only implicitly through $\alpha_\pm$.}  The purpose of this is to construct ``valence independent'' formulas for generating functions enumerating graphs embedded in surfaces.  By valence independent, we mean a formula that holds regardless of the allowed numbers of edges meeting at vertices.

%  %  %  %  %  %  %  %  %  %  %  %  %  %  %  %  %  %  %  %  %  %  %  %  %  %  %  %  
 %  %  %  %  %  %  %  %  %  %  %  %  %  %  %  %  %  %  %  %  %  %  %  %  %  %  %  %  

\subsection{Background on topological expansion of RM partition function}

In the context of random matrix combinatorics, a map of genus $g$ is an equivalence class of labeled graphs embedded in a genus $g$ oriented surface, such that the faces are topological discs.  This condition on the faces ensures that the graph is embedded in a surface of minimal genus.  Two maps are equivalent if an orientation preserving homeomorphism of the surface takes one map to the other, preserving labels.  

The labels a map is equipped with are slightly complicated, but necessary for the connection to random matrices.  For each $j$, let $k_j$ be the number of $j$-valent vertices of some map.  Each vertex of valence $j$ is assigned a unique label from the set $1,2,\ldots, k_j$.  Thus two vertices of the same valence must have different labels, but two vertices of different valences could have the same label.  Also, at each vertex one of the incident edges is marked; this marking of edges can be represented as a function $f$ from the vertex set to the edge set such that each edge $f(V)$ is incident to the vertex $V$.

The surprising connection between random matrices and maps is the topological expansion of the random matrix partition function \cite{BIZ}, \cite{EMcL03}.
\begin{align}
\log \frac{Z_N(t)}{Z_N(0)} \sim & N^2 e_0(t) + e_1(t) +N^{-2}e_2(t) +\ldots & \text{as }N\rightarrow \infty, \label{genus expansion} \\
e_g(x,t)=& \sum_{\Gamma\in \text{genus $g$ maps}} x^{\# \text{faces of }\Gamma} \prod_{j \geq 1} \frac{(-t_{j})^{\#\text{$j$-valent vertices of }\Gamma}}{(\#\text{$j$-valent vertices of }\Gamma)!}
\end{align} 
The sum on the right hand side of (\ref{genus expansion}) is divergent.  By ``$\sim$'' we mean that any partial sum $N^2e_0 + \ldots + N^{2-2g}e_g$ gives an approximation for the left hand side with error of order $o(N^{2-2g})$.

Bessis, Itzykson and Zuber stated the topological expansion (\ref{genus expansion}) in \cite{BIZ}, and showed that if such an expansion exists and can be differentiated term by term with respect to the parameters $t_j$, then the coefficients $e_g$ are generating functions for maps.  Ercolani and McLaughlin gave sufficient conditions on $t$ for the topological expansion to hold, and be differentiated term by term \cite{EMcL03}.  Their proof involves analyzing the Riemann-Hilbert problem for a family of orthogonal polynomials associated with the matrix ensemble.  For the measure $dP_N(M)$ to be normalizable as a density on Hermitian matrices it is necessary for the 
leading term of the potential to be $t_d \lambda^d$ with $d$ even and $t_d>0$.  However, we wish to consider potentials where $d$ is odd.  In the case $d=3$, Bleher and Dea\~{n}o considered an ensemble of unitarily diagonalizable matrices with eigenvalues on some curve in $\mathbb{C}$ such that $dP_N(M)$ can normalized to be a probability measure.  Using Riemann-Hilbert methods, they showed that the topological expansion holds in this case \cite{Bleher10}.  It is clear that their proof of the topological expansion can be extended to the case of a potential with an arbitrary leading term.

Bessis, Itzykson and Zuber gave formulas for $e_0, e_1$ and $e_2$ in the case of a potential $V(\lambda) = \frac{1}{2} \lambda^2 + t_4 \lambda^4$ \cite{BIZ}.  In this case, the functions $e_g$ enumerate 4-valent maps.  For example, they found that
\begin{align}
e_1 =& -\frac{1}{12} \log (2-z).
\end{align}
Recall that the function $z$ is related to the support of the equilibrium measure by $\alpha_\pm =u \pm 2\sqrt{z}$.  Note that in this case $u=0$ because the potential is an even function.
Ercolani, McLaughlin and Pierce \cite{EMcLP08} generalized the formulas for $e_g$ given in \cite{BIZ} to the case of a potential $V(\lambda)= \frac{1}{2}\lambda^2 + t_{2\nu}\lambda^{2\nu}$.  For example, in this case the generating function for maps on the torus is given by
\begin{align}
e_1 =& -\frac{1}{12} \log (\nu -(\nu-1)z). \label{EMcLP e1}
\end{align}
Their method for computing $e_g$ exploits the connection between random matrices and orthogonal polynomials.  In particular, they used a continuum limit of the string equations for recurrence coefficients of the orthogonal polynomials.

Ambjorn, Checkov, Kritjansen and Makeenko \cite{BeyondSphericalLimit} gave an alternate method for calculating the generating functions $e_g$.  Their approach relies on loop equations, which are closely related to Virasoro constraints for the random matrix partition function.  Eynard later streamlined the calculation of generating functions from loop equations \cite{EynardLoopEqs}.  In section 6 we will derive a formula for $e_1$ from loop equations and our valence independent formula for the equilibrium measure.  We rely on the exposition of loop equations given by Borot and Guionnet in \cite{GuionnetBeta}.

%  %  %  %  %  %  %  %  %  %  %  %  %  %  %  %  %  %  %  %  %  %  %  %  %  %  %  %  
 %  %  %  %  %  %  %  %  %  %  %  %  %  %  %  %  %  %  %  %  %  %  %  %  %  %  %  %  

\subsection{Results}

We derive valence independent formulas for the function $h(\lambda)$ appearing in (\ref{psiFormula}).  If the potential $V(\lambda)$ is an even function we find that
\begin{align} \label{hFormula}
h(\lambda) =&\sum_{k=0}^{d-2} \left(\lambda -2\sqrt{z} \right)^k \sum_{m=\lceil k/2 \rceil}^k \frac{2^{3m-2k}}{(2m+1)!!} \binom{m}{k-m} z^{m-k/2} \left( \frac{1}{z_x}\partial_x \right)^m \frac{1}{z_x} .
\end{align}
It will be clear from our calculations in the next several sections that by replacing $\sqrt{z}\mapsto -\sqrt{z}$ in (\ref{hFormula}), we obtain a formula for $h(\lambda)$ expanded about the left endpoint $\lambda=-2\sqrt{z}$ of the equilibrium measure.

To give a formula in the case of a general polynomial potential, we define a sequence of functions $\phi_m,\psi_m$ that play a role analogous to $(z_{x}^{-1}\partial_x)^m z_{x}^{-1}$ appearing in (\ref{hFormula}).
\begin{alignat}{2}
\vect{\phi_0}{\psi_0}=& \vect{0}{x}, \qquad & 
\vect{\phi_{m+1}}{\psi_{m+1}} =&  \frac{1}{z_{x}^2-z u_{x}^2} \matr{-z u_x }{z_x}{z z_x}{-z u_x} \partial_x  \vect{\phi_m}{\psi_m} 
\end{alignat}
Also, define sequences of coefficients $c^{(\phi)}_{k,m},c^{(\psi)}_{k,m}$ so that the following formula holds for all integers $k\geq 0$.
\begin{align}
 \frac{1}{(T-2+T^{-1})^{k+1}} 
=& \sum_{m =1}^{k+1}
c^{(\phi)}_{k,m} \frac{m!}{(T-T^{-1})^{2m}}\sum_{l=0}^m \binom{m}{l}^2 T^{2l-m}\nonumber  \\
&\qquad +c^{(\psi)}_{k,m} \frac{m!}{(T-T^{-1})^{2m}}\sum_{l=0}^{m-1} \binom{m-1}{l}\binom{m+1}{l+1} T^{2l+1-m} \label{c phi psi def}
\end{align}
It is shown in section 2.4 that such coefficients exist.

We will prove the following formula for $h(\lambda)$, valid for an arbitrary polynomial potential.
\begin{align}
h(\lambda) =& \sum_{k =0}^{d-2} (\lambda-u- 2\sqrt{z})^k 
\sum_{m =1}^{k+1} \sqrt{z}\,{}^{m-k-2}\left( c^{(\phi)}_{k,m} \sqrt{z}\phi_m +c^{(\psi)}_{k,m} \psi_m\right) \label{h formula}
\end{align}
By replacing $\sqrt{z}\mapsto -\sqrt{z}$ in (\ref{h formula}), one obtains a formula for $h(\lambda)$ centered at the left endpoint $u-2\sqrt{z}$ of the equilibrium measure.

Using (\ref{h formula}) and loop equations, we derive a formula for $e_1$. 
\begin{align}
e_1 =& \frac{1}{24} \log \frac{ (\partial_x u)^2 z - (\partial_x z)^2}{z^2} \label{e1 valence indep} 
\end{align}
This formula holds for an arbitrary polynomial potential $V(\lambda)$.  In section 7 we explain how this formula reduces to 
\begin{align}
e_1 =& \frac{1}{24} \log \frac{4-u^2/z}{(j-(j-2)z)^2-(j-2)^2u^2z}
\end{align}
in the special case of a potential $V(\lambda) = \frac{1}{2}\lambda^2 +t \lambda^j$.  This corresponds to the case where each vertex has the same valence $j$, thus extending formula (\ref{EMcLP e1}) of \cite{EMcLP08} to the case where $j$ is possibly odd.

%  %  %  %  %  %  %  %  %  %  %  %  %  %  %  %  %  %  %  %  %  %  %  %  %  %  %  %  
 %  %  %  %  %  %  %  %  %  %  %  %  %  %  %  %  %  %  %  %  %  %  %  %  %  %  %  %  

\subsection{Examples}
Formula (\ref{hFormula}) gives a series for $h(\lambda)$.
\begin{align}
h(\lambda) =& \frac{1}{z_x} 
 -\frac{2 \sqrt{z} z_{xx}}{3 z_{x}^3}  \left( \lambda-2\sqrt{z} \right)
+ \left( -\frac{ z_{xx}}{6 z_{x}^3}+ \frac{12 z  z_{xx}^2- 4z z_{x} z_{xx}}{15 z_{x}^5} \right)
\left(  \lambda-2\sqrt{z}  \right)^2 +\ldots  \label{h even example}
\end{align}
The sum is actually finite since $h(\lambda)$ has polynomial degree $d-2$.  Thus ignoring the ``$\ldots$'' in (\ref{h even example}) gives a formula for $h(\lambda)$ valid for potentials $V(\lambda) = (t_2 +\frac{1}{2})\lambda^2 + t_4 \lambda^4$.

For small indices, the coefficients appearing in formula (\ref{h formula}) are given in the following tables, calculated using $\emph{Mathematica}$.  The coefficients along the diagonal appear to be $2^k/(2k+1)!!$, but otherwise we are not able to give a general formula.
\begin{align}
\begin{array}{cc|ccccc}
c^{(\phi)}_{k,m}&&&&m&& \\ 
&&1&2&3&4&5 \\ \hline
 &0 & 1 & 0 &  0 & 0 & 0 \\
 &1 & 0 & \frac{2}{3} & 0 & 0 & 0 \\
 k&2 & 0 & -\frac{1}{30} & \frac{4}{15} & 0 & 0 \\
 &3 & 0 & \frac{1}{140} & -\frac{2}{105} & \frac{8}{105} & 0 \\
 &4 & 0 & -\frac{1}{630} & \frac{1}{252} & -\frac{2}{315} & \frac{16}{945}
\end{array}
\qquad
\begin{array}{cc|ccccc}
c^{(\psi)}_{k,m}&&&&m&& \\ 
&&1&2&3&4&5 \\ \hline
 &0 & 1 & 0 &  0 & 0 & 0 \\
 &1 & -\frac{1}{6} & \frac{2}{3} & 0 & 0 & 0 \\
 k&2 & \frac{1}{30} & -\frac{1}{10} & \frac{4}{15} & 0 & 0 \\
 &3 & -\frac{1}{140} & \frac{2}{105} & -\frac{4}{105} & \frac{8}{105} & 0 \\
 &4 &  \frac{1}{630} & -\frac{1}{252} & -\frac{1}{140} & -\frac{2}{189} &\frac{16}{945}
\end{array}
\nonumber
\end{align}
The first few functions $\phi_m,\psi_m$ are given below.  We calculated them using \emph{Mathematica}.
\begin{align}
\begin{array}{ >{\displaystyle}c | *3{>{\displaystyle}c}}
m & 0 & 1 & 2 \\
   \phi_m & 0 & \frac{z_x}{z_{x}^2-z u_{x}^2} & \frac{z u_{xx} z_{x}^3-z^2 u_{x}^3 z_{xx}+u_x z_{x}^4+z
   u_{x}^3 z_{x}^2-3 z u_x z_{x}^2 z_{xx}+3 z^2
   u_{x}^2 u_{xx} z_x}{\left(z
   u_{x}^2-z_{x}^2\right)^3}\\
   \psi_m & x & \frac{z u_x}{z u_{x}^2-z_{x}^2} & \frac{-2 z u_{x}^2 z_{x}^3+3 z^2 u_{x}^2 z_{x} z_{xx}-z^3
   u_{x}^3 u_{xx}-3 z^2 u_x u_{xx} z_{x}^2+z z_{x}^3 z_{xx}}{\left(z u_{x}^2-z_{x}^2\right)^3}
   \end{array}
   \nonumber
   \end{align}
If $V(\lambda)= \frac{1}{2}\lambda^2 +\sum_{j=1}^4 t_j \lambda^j$, then the polynomial $h(\lambda)$ is quadratic.  Formula (\ref{h formula}) becomes
\begin{align}
h(\lambda) =& 
 \left(\phi_1+\frac{\psi_1}{\sqrt{z}} \right)
+\left(\frac{2}{3}\left(\phi_2 +\frac{\psi_2}{\sqrt{z}}\right)-\frac{1}{6z}\psi_1  \right)(\lambda-u-2\sqrt{z}) \nonumber \\
& \qquad + \left(\frac{4}{15}\left(\phi_3+\frac{\psi_3}{\sqrt{z}}\right)  -\frac{1}{10\sqrt{z}}\left(\frac{\phi_2}{3}+\frac{\psi_2}{\sqrt{z}}\right) +\frac{\psi_1 }{30z^{3/2}}\right)\left(\lambda-u-2\sqrt{z}\right)^2 \\
=& \frac{1}{\sqrt{z}u_x + z_x} - \frac{\sqrt{z}u_{x}^2 + 3u_x z_x + 4z u_{xx} + 4\sqrt{z} z_{xx}}{6(\sqrt{z}u_x +z_x)^3}(\lambda-u-2\sqrt{z}) \nonumber \\
&  \qquad +\frac{\text{Polynomial with 16 terms}}{(\sqrt{z}u_x +z_x)^5}(\lambda-u-2\sqrt{z})^2.
\end{align}

%  %  %  %  %  %  %  %  %  %  %  %  %  %  %  %  %  %  %  %  %  %  %  %  %  %  %  %  
 %  %  %  %  %  %  %  %  %  %  %  %  %  %  %  %  %  %  %  %  %  %  %  %  %  %  %  %  
  %  %  %  %  %  %  %  %  %  %  %  %  %  %  %  %  %  %  %  %  %  %  %  %  %  %  %  %  
%  %  %  %  %  %  %  %  %  %  %  %  %  %  %  %  %  %  %  %  %  %  %  %  %  %  %  %  
 %  %  %  %  %  %  %  %  %  %  %  %  %  %  %  %  %  %  %  %  %  %  %  %  %  %  %  %  
  %  %  %  %  %  %  %  %  %  %  %  %  %  %  %  %  %  %  %  %  %  %  %  %  %  %  %  %  
%  %  %  %  %  %  %  %  %  %  %  %  %  %  %  %  %  %  %  %  %  %  %  %  %  %  %  %  
 %  %  %  %  %  %  %  %  %  %  %  %  %  %  %  %  %  %  %  %  %  %  %  %  %  %  %  %  
  %  %  %  %  %  %  %  %  %  %  %  %  %  %  %  %  %  %  %  %  %  %  %  %  %  %  %  %  
%  %  %  %  %  %  %  %  %  %  %  %  %  %  %  %  %  %  %  %  %  %  %  %  %  %  %  %  
 %  %  %  %  %  %  %  %  %  %  %  %  %  %  %  %  %  %  %  %  %  %  %  %  %  %  %  %  
  %  %  %  %  %  %  %  %  %  %  %  %  %  %  %  %  %  %  %  %  %  %  %  %  %  %  %  %  

\section{Formula for $h(z)$, general polynomial potential}

In this section we prove formula (\ref{h formula}).  
Expanding $h(\lambda)$ in a geometric series
we obtain
\begin{align}
h(\lambda)=& [y^{-1}]_{\infty} \frac{x V'(y)}{(y-\lambda)\sqrt{(y-\alpha_-)(y-\alpha_+)}} \nonumber \\
=& \sum_{k\geq 0} (\lambda-\alpha_+)^k I_k, \label{Ik 12345}
\end{align}
where $I_{k}$ is the following integral.
\begin{align}
I_k=&   \frac{1}{2\pi i}\oint_{\infty} \frac{ x V'(y)}{(y-\alpha_+)^{k+1}} \, \frac{dy}{\sqrt{(y-\alpha_-)(y-\alpha_+)}}
 \label{hCoeffEq1}
\end{align}
By $\oint_{\infty}$ we mean $\lim_{R \rightarrow\infty} \oint_{\partial B_R(0)}$; that is, the contour is a small circle around $\infty$ on the Riemann sphere.

Let us summarize the argument to follow in the next several subsections.  In section 2.1, we make a change of variables to express $I_{k}$ as the residue of a rational function at $\infty$; then assuming both the existence of $c^{(\phi)}_{k,m},c^{(\psi)}_{k,m}$ as defined in the introduction, and formulas for $\phi_m$ and $\psi_m$, we prove formula (\ref{h formula}).  In sections 2.2 and 2.3 we prove the formulas for $\phi_m$ and $\psi_m$ used in section 2.1.  In section 2.4 we prove that the coefficients $c^{(\phi)}_{k,m},c^{(\psi)}_{k,m}$ exist.

%  %  %  %  %  %  %  %  %  %  %  %  %  %  %  %  %  %  %  %  %  %  %  %  %  %  %  %  
 %  %  %  %  %  %  %  %  %  %  %  %  %  %  %  %  %  %  %  %  %  %  %  %  %  %  %  %  

\subsection{Uniformizing the square root in $I_{k}$}
We make a change of variables to uniformize the square root appearing in (\ref{hCoeffEq1}).  This change of variables was applied to a similar integral in \cite{DiF04} to derive the leading order string equations.  It reflects the Motzkin path structure of string equations, where $T,1,T^{-1}$ correspond to upward, flat and downward steps. 
\begin{alignat}{2}
y=&  \sqrt{z}\,T + u + \sqrt{z} \,T^{-1} &\qquad
y-\alpha_+ =&  \sqrt{z}\,T^{-1}(T-1)^2 \nonumber \\
\frac{dT}{T} =& \frac{dy}{\sqrt{(y-\alpha_-)(y-\alpha_+)}}  &
y-\alpha_- =& \sqrt{z}\, T^{-1}(T+1)^2  \label{uniformize}
\end{alignat}
We now apply this change of variables to the integral in (\ref{hCoeffEq1}).
\begin{align}
I_{k}=&  \frac{1}{2\pi i}\oint_{\infty} \frac{ x V'(y)}{(y-\alpha_+)^{k+1}} \, \frac{dy}{\sqrt{(y-\alpha_-)(y-\alpha_+)}} \nonumber \\
=&   \frac{1}{2\pi i} \oint_{\infty} \frac{x V'( \sqrt{z}\,T + u + \sqrt{z}\, T^{-1})}{\sqrt{z}^{k+1}(T-2 + T^{-1})^{k+1}} \, \frac{dT}{T} \label{Ik formula1}
\end{align}
Let us explain the following calculation.  In formula (\ref{Ik formula1}) we express the integral as a Laurent coefficient at infinity; on the next line we apply the definition (\ref{c phi psi def}) of the coefficients $c^{(\phi)}_{k,m},c^{(\psi)}_{k,m}$; finally we apply formulas for $\phi_m,\psi_m$ which will be proven later in this section.
\begin{align}
I_k=& [T^0]_{\infty} \frac{ x V'(\sqrt{z}\, T + u + \sqrt{z}\, T^{-1})}{\sqrt{z}^{k+1}(T-2+ T^{-1})^{k+1}} \nonumber  \\
=& [T^0]_{\infty} \sum_{m =1}^{k+1}
c^{(\phi)}_{k,m} \frac{ x V'(\sqrt{z}\, T + u + \sqrt{z}\, T^{-1})m!}{\sqrt{z}^{k+1}(T-T^{-1})^{2m}}\sum_{l=0}^m \binom{m}{l}^2 T^{2l-m}\nonumber  \\
&\qquad \qquad \qquad +c^{(\psi)}_{k,m} \frac{x V'(\sqrt{z}\, T + u + \sqrt{z}\, T^{-1})m!}{\sqrt{z}^{k+1}(T-T^{-1})^{2m}}\sum_{l=0}^{m} \binom{m-1}{l}\binom{m+1}{l+1} T^{2l+1-m}  \nonumber \\
= & \sum_{m=1}^{k+1} \sqrt{z}^{m-k-1} c^{(\phi)}_{k,m} \phi_{m} + \sqrt{z}^{m-k-2} c^{(\psi)}_{k,m} \psi_m \label{Ik last formula}
\end{align}
Equation (\ref{Ik last formula}) is obviously equivalent to our formula (\ref{h formula}) for the function $h(\lambda)$.  
Lemma 3 in section 2.4 will justify the last line of the above calculation.

%  %  %  %  %  %  %  %  %  %  %  %  %  %  %  %  %  %  %  %  %  %  %  %  %  %  %  %  
 %  %  %  %  %  %  %  %  %  %  %  %  %  %  %  %  %  %  %  %  %  %  %  %  %  %  %  %  

\subsection{Residue formulas for $\phi_m$ and $\psi_m$}
The functions $\phi_m, \psi_m$ are defined by
\begin{alignat}{2}
\vect{\phi_0}{\psi_0}=& \vect{0}{x}, \qquad & 
\vect{\phi_{m+1}}{\psi_{m+1}} =&  \frac{1}{z_{x}^2-z u_{x}^2} \matr{-z u_x }{z_x}{z z_x}{-z u_x} \partial_x  \vect{\phi_m}{\psi_m}. \label{recursion}
\end{alignat}
Using the change of variables (\ref{uniformize}) in equations (\ref{endPts1}) and  (\ref{endPts2}), we obtain the following restatements of the equations for the end points of support of the equilibrium measure.
\begin{align}
0=& [T^0]_{\infty}\, V'(T+u+zT^{-1}) \label{altEndPts1}\\
1=& [T^0]_{\infty}\, TV'(T+u+zT^{-1}) \label{altEndPts2}
\end{align}
Notice that the expression $V'(T+u+zT^{-1})$ depends on $x$ both explicitly through the prefactor $1/x$ in the definition of $V$, and implicitly through $u$ and $z$.  For the rest of this section it will be convenient to multiply through by $x$ to eliminate explicit $x$ dependence in the potential.
In this section we prove the following formula.

\begin{lemma} The functions $\phi_m,\psi_m$ have the following integral representations.
\begin{align}
 \vect{\phi_m}{\psi_m} =& [T^0]_{\infty}\,\vect{ xV^{(m+1)}(T+u +zT^{-1})}{TxV^{(m+1)}(T+u+zT^{-1})} \label{target 2}
\end{align}
\end{lemma}
\begin{proof}
We proceed by induction on $m$.  The base case $m=0$ immediate from the definitions of $\phi_0,\psi_0$ and equations (\ref{altEndPts1}), (\ref{altEndPts2}).  For the inductive step, introduce the notation $Y=T+u+zT^{-1}$ and differentiate with respect to $x$ in (\ref{target 2}).
\begin{align}
\partial_x \vect{\phi_m}{\psi_m} =& [T^0]_{\infty}\, \vect{xV^{(m+1)}(Y)(u_x + z_x T^{-1})}{ xV^{(m+1)}(Y)(u_x T + z_x )}  \label{rec b1}
\end{align}
Multiplying on the left by the matrix appearing in (\ref{recursion}), we obtain 
\begin{align}
 \matr{-z u_x }{z_x}{z z_x}{-z u_x} \partial_x  \vect{\phi_m}{\psi_m} =&
 [T^0]_{\infty} \vect{ xV^{(m+1)}(Y)(-zu_{x}^2 +z_{x}^2 +u_x z_x (T-zT^{-1}) }
{xV^{(m+1)}(Y)(z z_{x}^2 T^{-1} -z u_{x}^2 T)} \\
=& [T^0]_{\infty} \vect{ xV^{(m+1)}(Y)(-zu_{x}^2 +z_{x}^2) }
{xV^{(m+1)}(Y)(z_{x}^2 -z u_{x}^2) T} \label{lemma1 eq 3}
\end{align}
The simplification on the right hand side came from applying the identity below (with $p=1$).
\begin{align}
\intinfty Y^j T^p \frac{dT}{T} =& \intinfty Y^j \frac{z^p}{T^p} \, \frac{dT}{T}
\end{align}
This identity can be proven by changing variables $T\mapsto z/T$, and then observing that the integrand has the same index zero Laurent coefficient at $\infty$ as at $0$.  
\end{proof}

%  %  %  %  %  %  %  %  %  %  %  %  %  %  %  %  %  %  %  %  %  %  %  %  %  %  %  %  
 %  %  %  %  %  %  %  %  %  %  %  %  %  %  %  %  %  %  %  %  %  %  %  %  %  %  %  %  

\subsection{Second formula for $\phi_m$ and $\psi_m$}
\begin{lemma} The following identities hold for all integers $m\geq 0$.
\begin{align}
\left(-\partial_T \circ \frac{1}{1-T^{-2}}\right)^m \frac{1}{T} =& \frac{m!}{(T-T^{-1})^{2m}} \sum_{l=0}^m \binom{m}{l}^2 T^{2l-m-1}  \label{c ident 1} \\
\left(-\partial_T \circ \frac{1}{1-T^{-2}}\right)^m 1 =& \frac{m!}{(T-T^{-1})^{2m}} \sum_{l=0}^m \binom{m-1}{l}\binom{m+1}{l+1} T^{2l-m}  \label{c ident 2}
\end{align}
\end{lemma}
\begin{proof}
We prove (\ref{c ident 1}) by induction on $m$.  The base case $m=0$ is trivial.  The inductive step is a direct calculation.
\begin{align}
& -\partial_T \left(\frac{1}{1-T^{-2}} \frac{m!}{(T-T^{-1})^{2m}} \sum_{l=0}^m \binom{m}{l}^2 T^{2l-m-1}  \right)\nonumber  \\
& \qquad = \frac{m!}{(T-T^{-1})^{2m+2}} \sum_{l=0}^m T^{2l-m} \left( (3m-2l+1)\binom{m}{l}^2 +(2l+m+3)\binom{m}{l+1}^2 \right) \nonumber  \\
& \qquad =\frac{m!}{(T-T^{-1})^{2m+2}}  \sum_{l=0}^m T^{2l-m} (m+1)\binom{m+1}{l+1}^2
\end{align}
The last step of the above calculation can be verified by using Pascal's identity to express all the binomial coefficients as multiples of $\binom{m}{l}$.  An analogous proof can be given for (\ref{c ident 2}).
\end{proof}

\begin{lemma} Repeated integration by parts in formulas (\ref{target 2}) yields the following formulas.
\begin{align}
\phi_m =& [T^0]_{\infty} \frac{\sqrt{z}^{-m}xV'(\sqrt{z} \,T+u+\sqrt{z}\,T^{-1})m!}{(T-T^{-1})^{2m}}\sum_{l=0}^m \binom{m}{l}^2 T^{2l-m} \\
\psi_m =& [T^0]_{\infty} 
\frac{\sqrt{z}^{1-m}xV'(\sqrt{z} \,T+u+\sqrt{z}\,T^{-1})m!}{(T-T^{-1})^{2m}}\sum_{l=0}^m \binom{m-1}{l}\binom{m+1}{l+1} T^{2l-m+1}
\end{align}
\end{lemma}
\begin{proof}
The operator $[T^0]_\infty$ is an integral along a closed contour, so we may integrate by parts with no boundary terms.  For each integration by parts we multiply and divide by $(1-T^{-2})$ so that we may use 
\begin{align}
\int V^{(r+1)}(\sqrt{z}T + u + \sqrt{z}T^{-1})(1-T^{-2})\, dT =& \sqrt{z}^{-1}V ^{(r)}(\sqrt{z}T + u + \sqrt{z}T^{-1}). 
\end{align}
The result of applying integration by parts $m$ times in this way is thus given by the previous lemma.
\end{proof}

\subsection{Existence of the coefficients $c^{(\phi)}_{k,m}$ and $c^{(\psi)}_{k,m}$}

To complete the proof of our formula for $h(\lambda)$ we must show that there exist coefficients $c^{(\phi)}_{k,m},c^{(\psi)}_{k,m}$ so that (\ref{c phi psi def}) holds.  Define the following linear spaces.
\begin{align}
V_k = \{ & \text{rational functions $f(T)$ such that }   f(\infty)=0, \nonumber \\
& \text{and }(T-1)^k(T+1)^k f(T) \text{is an entire function} \}, \nonumber
\end{align}
Also define the following two sequences of functions.
\begin{align}
\widetilde{\phi}_m =& \frac{m!}{(T-T^{-1})^{2m}}\sum_{l=0}^m \binom{m}{l}^2 T^{2l-m} \\
\widetilde{\psi}_m=& \frac{m!}{(T-T^{-1})^{2m}}\sum_{l=0}^{m} \binom{m-1}{l}\binom{m+1}{l+1} T^{2l+1-m}.
\end{align}
We claim that $\{\widetilde{\phi}_m, \widetilde{\psi}_m\}_{m=1}^k$ span $V_k$.  Since the dimension of $V_k$ is $2k$, it suffices to show that they are linearly independent.  The leading coefficients of the Laurent expansions of $\widetilde{\phi}_m,\widetilde{\psi}_m$ about $\pm 1$ can be computed explicitly using the identities
\begin{align}
\sum_{l=0}^m \binom{m}{l}^2 =& \binom{2m}{m}, \label{c ident 24a} \\
 \sum_{l=0}^{m-1} \binom{m-1}{l-1} \binom{m+1}{l+1} =& \binom{2m}{m}, \label{c ident 24b} \\
\sum_{l=0}^m \binom{m}{l}^2 (-1)^l =& \begin{cases} (-1)^{m/2} \binom{m}{m/2} & \text{ if $m$ even} \\ 0 & \text{ if $m$ odd} \end{cases}, \label{c ident 24c} \\
 \sum_{l=0}^{m-1} \binom{m-1}{l-1} \binom{m+1}{l+1} (-1)^l =& \begin{cases} 0 & \text{ if $m$ even} \\ (-1)^{(m-1)/2}2^m \frac{(m-2)!!}{(m-1)!!} & \text{ if $m$ odd} \end{cases} \label{c ident 24d}.
\end{align}
The above formulas are all easily proven by induction on $m$.  Formulas (\ref{c ident 24a}), (\ref{c ident 24b}) show that $\widetilde{\phi}_m,\widetilde{\psi}_m$ have nonvanishing Laurent coefficients of order $2m$ at $T=1$, thus neither function is in the space $V_{k-1}$.  Formulas (\ref{c ident 24c}), (\ref{c ident 24d}) show that $\widetilde{\phi}_m,\widetilde{\psi}_m$ have different Laurent coefficients at $T=-1$, and thus are linearly independent.
Existence of $c^{(\phi_)}_{k,m},c^{(\psi)}_{k,m}$ is now clear since and formula (\ref{c phi psi def}) expresses $(T-2+T^{-1})^{-k-1}$ in the basis $\{\widetilde{\phi}_m,\widetilde{\psi}_m\}_{m=1}^{k+1}$ for the space $V_{k+1}$.

%  %  %  %  %  %  %  %  %  %  %  %  %  %  %  %  %  %  %  %  %  %  %  %  %  %  %  %  
 %  %  %  %  %  %  %  %  %  %  %  %  %  %  %  %  %  %  %  %  %  %  %  %  %  %  %  %  
  %  %  %  %  %  %  %  %  %  %  %  %  %  %  %  %  %  %  %  %  %  %  %  %  %  %  %  %  
%  %  %  %  %  %  %  %  %  %  %  %  %  %  %  %  %  %  %  %  %  %  %  %  %  %  %  %  
 %  %  %  %  %  %  %  %  %  %  %  %  %  %  %  %  %  %  %  %  %  %  %  %  %  %  %  %  
  %  %  %  %  %  %  %  %  %  %  %  %  %  %  %  %  %  %  %  %  %  %  %  %  %  %  %  %  
%  %  %  %  %  %  %  %  %  %  %  %  %  %  %  %  %  %  %  %  %  %  %  %  %  %  %  %  
 %  %  %  %  %  %  %  %  %  %  %  %  %  %  %  %  %  %  %  %  %  %  %  %  %  %  %  %  
  %  %  %  %  %  %  %  %  %  %  %  %  %  %  %  %  %  %  %  %  %  %  %  %  %  %  %  %  
%  %  %  %  %  %  %  %  %  %  %  %  %  %  %  %  %  %  %  %  %  %  %  %  %  %  %  %  
 %  %  %  %  %  %  %  %  %  %  %  %  %  %  %  %  %  %  %  %  %  %  %  %  %  %  %  %  
  %  %  %  %  %  %  %  %  %  %  %  %  %  %  %  %  %  %  %  %  %  %  %  %  %  %  %  %  

\section{Formula for $h(z)$, even polynomial potential}
In the case that the potential is an even function we have $u=0$.  Thus equation (\ref{Ik formula1}) becomes
\begin{align}
h(\lambda) =&\sum_{k\geq 0} (\lambda-2\sqrt{z})^k I_k \label{even h 1} \\ 
I_k=& [T^0]_{\infty}  \frac{xV'\left(\sqrt{z}\, T  +\sqrt{z}\, T^{-1}\right)}{\sqrt{z}^{k+1}(T-2+ T^{-1})^{k+1}}. \nonumber  
\end{align}
Our formula for $h(\lambda)$ for even potentials does not obviously follow from the formula for general potentials.  This is because we do not have closed form formulas for $c^{(\phi)}_{k,m},c^{(\psi)}_{k,m}$.

Let us explain the following calculation, which outlines our proof of formula (\ref{hFormula}).
First we project the integrand onto the even part of its Laurent series at infinity; this does not change the value of the integral.  
Since the $V'(\lambda)$ is an odd function, this amounts to taking the odd part of $(T-2+T^{-1})^{-k-1}$ at infinity.  
Next we apply a combinatorial identity which we will state and prove as lemma 4 in section 3.2.
Finally we apply a residue formula which will be proven in section 3.1.
\begin{align}
I_k =& [T^0]_{\infty} \sqrt{z}^{-k-1}xV'(\sqrt{z}\,T+\sqrt{z}\,T^{-1}) \frac{T^{k+1}}{2}\left( \frac{1}{(T-1)^{2k+2}} + \frac{(-1)^k}{(T+1)^{2k+2}}\right)\nonumber \\  
=& [T^0]_{\infty} \sqrt{z}^{-k-1}xV'(\sqrt{z}\,T+\sqrt{z}\,T^{-1}) \sum_{m=0}^k 2^{4m-2k}\binom{m}{k-m} \frac{T+T^{-1}}{(T-T^{-1})^{2m+2}} \label{1357} \\
=& \sum_{m=0}^k 2^{4m-2k} \binom{m}{k-m} \sqrt{z}^{-k-1} [T^0]_{\infty} xV'(\sqrt{z}\,T+\sqrt{z}\,T^{-1}) \frac{T+T^{-1}}{(T-T^{-1})^{2m+2}}  \nonumber \\
=&\sum_{m=0}^k 2^{4m-2k} \binom{m}{k-m} \sqrt{z}^{-k-1} \frac{\sqrt{z}^{2m+1}}{2^m(2m+1)!!}  \left(\frac{1}{z_x}\partial_x\right)^m\frac{1}{z_x} \label{2468}
\end{align}

Thus to complete the proof of formula (\ref{hFormula}) it suffices to prove formulas in sections 3.1 and 3.2 to justify lines (\ref{1357}) and (\ref{2468}), respectively.

\subsection{Residue formula justifying equality (\ref{1357})}
In the special case of an even potential, we have $\alpha_\pm =- \alpha_\mp$; thus we need only the single equation (\ref{endPts2}) to determine the endpoints of support of the equilibrium measure, which specializes ($u=0$) to
\begin{align}
x=& \,[y^0]_{\infty} \, \frac{y^2 xV'(y)/2}{\sqrt{y^2-4z}}.
\end{align}
We easily obtain (by induction on $m$) a formula for the result of applying $(\frac{1}{z_x}\partial_x)^m$ to both sides.  We then making a change of variables $y=\sqrt{z}\,T+ \sqrt{z}\,T^{-1}$; the resulting formula justifies equation (\ref{1357}).
\begin{align}
\left(\frac{1}{z_x} \partial_x\right)^m x =& 2^m (2m-1)!!\, [y^0]_\infty \, \frac{y^2 xV'(y)/2}{(y^2-4z)^{m+1/2}} \\
=& 2^{m-1}(2m-1)!!\sqrt{z}^{1-2m}\, [T^{0}]_{\infty}\, xV'(\sqrt{z}\,T+\sqrt{z}\, T^{-1}) \frac{T+T^{-1}}{(T-T^{-1})^{2m}}
\end{align}

\subsection{Combinatorial identity justifying equality (\ref{2468})}

\begin{lemma}
The following identity holds for all integers $k\geq 0$.
\begin{align}
\frac{T^{k+1}}{2}\left(\frac{1}{(T-1)^{2k+2}} + \frac{(-1)^k }{(T+1)^{2k+2}}\right)
=& \sum_{m=0}^k 2^{4m-2k} \binom{m}{k-m} \frac{T+T^{-1}}{(T-T^{-1})^{2m+2}} \label{c proof eq 1}
\end{align}
\end{lemma}
\begin{proof}

Consider the case that $k$ is odd. 
We let $2n+1=k$ and $q+n=m$, then clear denominators.  Equation (\ref{c proof eq 1}) becomes
\begin{align}
(T+1)^{4n+2}+(T-1)^{4n+2} =& 2(T^2+1) \sum_{q=0}^n \binom{n+q}{n-q}(4T)^{2q}(T^2-1)^{2n-2q}.
\end{align}
One easily checks that the left and right hand sides are equal for $n=0,1$.  The coefficients of the polynomials on the left hand side are the entries of Pascal's triangle with even row and column indices.  It is trivial to prove by induction that the left hand side satisfies the recurrence
\begin{align}
\text{LHS}(n+2) -2(1+6T^2+T^4) \text{LHS}(n+1) +(T^2-1)^4 \text{LHS}(n) =0
\end{align}
Thus it suffices to show that the right hand side also satisfies this recurrence.  We proceed using the method of summation by elimination (see section 4.1 of Aigner's book \cite{AignerBook}).  Let $N:n\mapsto n+1$ and $Q: q\mapsto q+1$ be the shift operators corresponding to $n,q$.  Define
\begin{align}
F(n,q) =& (T^2+1)\binom{n+q}{n-q} (4T)^{2q}(T^2-1)^{2n-2q}
\end{align}
We find the following relations on the function $F$
\begin{align}
0=& \left[(n-q-1)N - (n+q +1)(T^2-1)^2 \right]F(n,q)  \tag{a}\\
0\equiv & \left[ Q(2q)(2q-1)(T^2-1)^2 -(n+q+1)(n-q)16T^2 \right]F(n,q) & \mod Q-I \tag{b}
\end{align}
By congruent mod $Q-I$ we mean that the difference is divisible on the left by $Q-I$.
Let $\mathcal{A},\mathcal{B}$ be the linear operators acting on $F$ in equations (a), (b).  We take a linear combination of (a) and (b) to eliminate $q$.  By direct calculation, we find that
\begin{align}
&2(T^4+2T^2+1)k \,\mathcal{A} + \frac{1}{2}((T^2-1)^2+N)\,\mathcal{B} \nonumber \\
& \equiv -8(1+n)T^2(N(2+n) + n(T^2-1)^2) & \mod Q-I \tag{c} \\
&\qquad + q\big[N (T^4+14T^2 +1 +2n(T^4 +2T^2 +1)) \nonumber \\
& \qquad \qquad -(T^2-1)^2\left(3(T^2-1)^2 +2n (T^4 + 2T^2 +1)\right)\big] \nonumber  
\end{align}
This combination has eliminated the $q^2$ term from $\mathcal{B}$.  Let $\mathcal{C}$ be the right hand side of (c), and take a further linear combination to completely eliminate $q$.
\begin{align}
& \big[(n+1)(n+2)2(T^4+2T^2+1)((T^2-1)^2-N) \\
& \qquad +(T^4-10T^2+1)((n+1)N + (n+2)(T^2-1)^2) \big]\, \mathcal{A} \nonumber \\
& -[(n+1)N+(n+2)(T^2-1)^2]\mathcal{C} \nonumber \\
& \equiv -(6+13n+9n^2+2n^3)(T^2-1)^2 \left[N^2-2(1+6T^2 +T^{4})N +(T^2-1)^4\right] & \mod Q-I, \nonumber
\end{align}
By proposition 4.1 of \cite{AignerBook}, it follows that $N^2-2(1+6T^2 +T^{4})N +(T^2-1)^4$ is a recurrence operator for the sum $\sum_q F(n,q)$.  This proves our lemma in the case that $k$ is odd; an analogous proof holds in the case that $k$ is even.
\end{proof}

%  %  %  %  %  %  %  %  %  %  %  %  %  %  %  %  %  %  %  %  %  %  %  %  %  %  %  %  
 %  %  %  %  %  %  %  %  %  %  %  %  %  %  %  %  %  %  %  %  %  %  %  %  %  %  %  %  
  %  %  %  %  %  %  %  %  %  %  %  %  %  %  %  %  %  %  %  %  %  %  %  %  %  %  %  %  
%  %  %  %  %  %  %  %  %  %  %  %  %  %  %  %  %  %  %  %  %  %  %  %  %  %  %  %  
 %  %  %  %  %  %  %  %  %  %  %  %  %  %  %  %  %  %  %  %  %  %  %  %  %  %  %  %  
  %  %  %  %  %  %  %  %  %  %  %  %  %  %  %  %  %  %  %  %  %  %  %  %  %  %  %  %  
%  %  %  %  %  %  %  %  %  %  %  %  %  %  %  %  %  %  %  %  %  %  %  %  %  %  %  %  
 %  %  %  %  %  %  %  %  %  %  %  %  %  %  %  %  %  %  %  %  %  %  %  %  %  %  %  %  
  %  %  %  %  %  %  %  %  %  %  %  %  %  %  %  %  %  %  %  %  %  %  %  %  %  %  %  %  
%  %  %  %  %  %  %  %  %  %  %  %  %  %  %  %  %  %  %  %  %  %  %  %  %  %  %  %  
 %  %  %  %  %  %  %  %  %  %  %  %  %  %  %  %  %  %  %  %  %  %  %  %  %  %  %  %  
  %  %  %  %  %  %  %  %  %  %  %  %  %  %  %  %  %  %  %  %  %  %  %  %  %  %  %  %  

\section{Formula for $e_g$ in terms of the first correlator}

The purpose of this section is to derive a formula for $\partial_x e_g$ in terms of asymptotics of the first correlator for our matrix ensemble $dP_N(M)$.  We introduce necessary notations $\rho^{(1)}_N(\lambda),W_{1}^{(2g-1)}(y)$ and $\widetilde{\psi}(\lambda)$; then we state the formula for $\partial_x e_g$.

The $N^{th}$ one point correlation function $\rho^{(1)}_N(\lambda)$ is the marginal density for the distribution for a single eigenvalue sampled from a $dP_N$ random matrix 
\begin{align}
\rho^{(1)}_N(\lambda_1)=& \int_{\mathbb{R}^{N-1}}dP_N(\lambda_1,\ldots,\lambda_N)\, d\lambda_2 \ldots d\lambda_N.
\end{align}
As $N\rightarrow\infty$, the one point correlation functions converge to the equilibrium measure \cite{DeiftOPs}.  But the difference $\psi(\lambda)-\rho^{(1)}_N(\lambda)$ is rapidly oscillating in $N$.  For this reason, the one point correlation functions do not have an asymptotic expansion in powers of $N^{-1}$ (an expansion exists, but it has a more complicated form \cite{EMcL03}).  It is convenient to work with random matrix ``correlators'' which do have expansions in powers of $N^{-1}$.  The first correlator $W_{1,N}$ and its asymptotic coefficients $W_{1}^{(2g-1)}$ are defined by
\begin{align}
W_{1,N}(y)=& N \int \frac{\rho^{(1)}_N(\lambda)}{y-\lambda}\, d\lambda \\
\sim &  N W_{1}^{(1)}(y) + N^{-1} W_{1}^{(-1)}(y) + N^{-3} W_{1}^{(-3)}(y) + \ldots
\end{align}
Define an analytic function $\widetilde{\psi}$ with domain $\mathbb{C}\backslash [\alpha_-,\alpha_+]$ by
\begin{align}
\widetilde{\psi}_{+}(\lambda) =& \psi(\lambda) & \alpha_- < \lambda < \alpha_+ \label{psitilde def},
\end{align}
where the subscript $+$ indicates a boundary value from the upper half plane.  Equation (\ref{psitilde def}) defines $\widetilde{\psi}$ to be an analytic continuation of $\psi$.  

We now state the aforementioned formula for $\partial_x e_g$.  If $g\geq 1$ then
\begin{align}
\partial_x e_g =& C_g(x) - \frac{1}{x} \oint_C \widetilde{\psi}(y) \int_{y}^\infty W_{1}^{(1-2g)}(y')\, dy' \, dy. \label{e1eq1}
\end{align} 
The functions $C_g(x)$ can be thought of as constants since they do not depend on $t$.  Since the functions $e_g(t)$ are generating functions, we have an identity $e_g(t=0)=0$;  this expresses the fact that a map must have at least one vertex.  Thus one can calculate the constant $C_g(x)$ in (\ref{e1eq1}) by evaluating both sides at $t=0$.

\subsection{Outline of $\partial_x e_g$ calculation}
We will need to restrict a domain of integration from $\mathbb{R}$ to some finite interval.  Let $a_{+}= \alpha_{+} + 1$ and $a_-= \alpha_- -1$.  The ``1'' is arbitrary in the sense that we could have used any positive constant.  We now give an outline of our proof of formula (\ref{e1eq1}).  No step of the following calculation is immediate, and each will be justified in its own subsection.
\begin{align}
\frac{x}{N^2} \partial_x \log Z_N -V(a_+) \sim & -\int_{a_-}^{a_+} V'(\lambda) \irho \, d\lambda \label{42eq} \\
=& -\oint_\infty \widetilde{\psi}(y) \int_{a_-}^{a_+}  \frac{ \irho }{\lambda -y} \, d\lambda \, dy \label{43eq} \\
\sim& - \oint_\infty \widetilde{\psi}(y)\left[ \log\left(1-a_+/y\right)+ \int_{y}^\infty \int_{\mathbb{R}} \frac{\lambda \psi(\lambda)/y'}{y'-\lambda} \, d\lambda\, dy'\right] \, dy \nonumber \\
& -  \sum_{g\geq 1}N^{-2g} \oint_\infty \widetilde{\psi}(y) \int_{y}^\infty  W^{(1-2g)}_1(y') \, dy' \, dy \label{44eq}
\end{align}
Lines (\ref{42eq})-(\ref{44eq}) of the above calculation are justified in sections 4.2-4.4, respectively.  In section 4.5 we will obtain (\ref{e1eq1}) by extracting coefficients from both sides of (\ref{44eq}).

\subsection{Justification of formula (\ref{42eq})}

The partition function is defined by
\begin{align}
Z_N =& \int_{\mathbb{R}^N} \exp\left(-N \sum_{i=1}^N V(\lambda_i)\right)\, \Delta^2(\lambda) \, d^N\lambda.
\end{align}
Recall that $V(\lambda)$ depends on $x$.  Differentiating with respect to $x$ we obtain a relation between the partition function and the one point correlation function.
\begin{align} 
\partial_x \log Z_N =& \frac{1}{Z_n} \int_{\mathbb{R}^N} \left(\frac{N}{x}\sum_{i=1}^N V(\lambda_i)\right) e^{-N\sum_{i=1}^N V(\lambda_i)}\Delta^2 \lambda \, d^N \lambda \\
=& \frac{N^2}{x} \int_{\mathbb{R}} V(\lambda)\rho^{(1)}_N(\lambda)\, d\lambda \label{dx Zn 1}
\end{align}

  Because the one point correlation functions are decaying uniformly exponentially in $N$ for $\lambda$ in any set bounded away from the support of the equilibrium measure, the contribution to the integral below from $\mathbb{R}\backslash [a_-,a_+]$ is beyond all asymptotic orders.  It follows that

\begin{align}
\frac{x}{N^2}\partial_x \log Z_N 
\sim & \int_{a_-}^{a_+}  V(\lambda) \rho^{(1)}_N(\lambda)\, d\lambda \nonumber  \\
=&  \left. V(\lambda) \irho  \right|_{\lambda=a_-}^{a_+} - \int_{a_-}^{a_+}  V'(\lambda) \irho \, d\lambda \nonumber  \\
\sim &  V(a_+) - \int_{a_-}^{a_+} V'(\lambda) \irho \, d\lambda. \label{Zeq1}
\end{align}
For the last step we used the fact that 
\begin{align}
\int_{a_-}^{a_+} \rho^{(1)}(\lambda)\, d\lambda =& 1 + O(e^{-CN})
\end{align}
for some constant $C>0$.

\subsection{Applying the variational equality}
In this section we prove formula (\ref{43eq}).  We use the following formula relating the the potential and the equilibrium measure: for all $\lambda\in \mathbb{C}$ we have
\begin{align}
 V'(\lambda) =& \int_{\infty} \frac{\widetilde{\psi}(y)}{ \lambda-y } \, dy. \label{inverse eq meas}
\end{align}
This formula is an immediate consequence of formula (2.6) in \cite{Bertola BCurves}.  However, since (\ref{inverse eq meas}) seems to appear less frequently in the literature than other characterizations of the equilibrium measure, we give a derivation.  
\begin{proof}[Proof of formula (\ref{inverse eq meas})]
First suppose $\alpha_- < \lambda < \alpha_+$.  Starting with the variational equality for the equilibrium measure, we use $\widetilde{\psi}_+(\lambda)=-\widetilde{\psi}_-(\lambda)$ to rewrite the integral from $\alpha_-$ to $\alpha_+$ as a contour integral.  Then we move deform the contour to a small loop around $\infty$, picking up the residue at $\lambda+i\epsilon$ as we go.
\begin{align}
 V'(\lambda) =& \left( \lim_{\epsilon \uparrow 0} +\lim_{\epsilon \downarrow 0} \right) \int_{\alpha_-}^{\alpha_+} \frac{\psi(y)}{\lambda+ i\epsilon -y}\, dy \nonumber \\
=& \left( \lim_{\epsilon \uparrow 0} +\lim_{\epsilon \downarrow 0} \right) \frac{-1}{2} \oint_{C_\epsilon} \frac{\widetilde{\psi}(y)}{\lambda+i \epsilon-y}\, dy \nonumber \\
=&\left( \lim_{\epsilon \uparrow 0} +\lim_{\epsilon \downarrow 0} \right) \frac{-1}{2} \left( -2\pi i \widetilde{\psi}(\lambda + i\epsilon)- \oint_{\infty}\frac{\widetilde{\psi}(y)}{\lambda+i \epsilon-y}\, dy \right) \nonumber  \\ 
=&\int_{\infty} \frac{\widetilde{\psi}(y)}{ \lambda- y} \, dy
\end{align}
By our assumption that $\alpha_- < \lambda < \alpha_+$, the $\widetilde{\psi}(\lambda +i \epsilon)$ terms from each limit cancel.  Equality for all $\lambda\in \mathbb{C}$ follows by analytic continuation since both sides of (\ref{inverse eq meas}) are entire functions.
\end{proof}

We now use (\ref{inverse eq meas}) in formula (\ref{Zeq1}).  

\begin{align}
\frac{x}{N^2} \partial_x \log Z_N - V(a_+)=&
- \int_{a_-}^{a_+} V'(\lambda) \irho \, d\lambda \nonumber \\
=& -\int_{a_-}^{a_+} \oint_{\infty} \frac{\widetilde{\psi}(y)}{\lambda -y}\, dy \irho \, d\lambda \nonumber \\
=&  -\oint_\infty \widetilde{\psi}(y) \int_{a_-}^{a_+}  \frac{ \irho }{\lambda -y} \, d\lambda \, dy \label{Zeq2}
\end{align}
This proves formula (\ref{43eq}).

\subsection{Commuting an integral and a Cauchy transform}
We want to express (\ref{Zeq2}) in terms of the first correlator $W_1$.
%(which is like the Cauchy transform of $\rho^{(1)}$)
The Cauchy transform commutes with differentiation, so we calculate the ``price'' of moving commuting the indefinite integral and Cauchy transform in (\ref{Zeq2}).
\begin{align}
\partial_y \int_{a_-}^{a_+} \frac{\irho}{\lambda-y}\, d\lambda 
=& \int_{a_-}^{a_+} -\partial_\lambda \left( \frac{1}{\lambda-y}\right) \irho \, d\lambda \nonumber \\
=&  \frac{\int_{a_-}^{a_+} \rho^{(1)}_N}{y- a_+} + \int_{a_-}^{a_+} \frac{\rho^{(1)}_N(\lambda) }{\lambda-y}\, d\lambda \nonumber \\
\sim & \frac{1}{y-a_+} + \int_{\mathbb{R}} \frac{\rho^{(1)}_N(\lambda)}{\lambda-y} \, d\lambda \label{cRho1}
\end{align}
We integrate (\ref{cRho1}) with $y$ going from some $C>b$ to $y$.  Preparing to take a limit $C\rightarrow\infty$, we state an identity which will group terms for cancellation as $C\rightarrow\infty$.
\begin{align}
\log\left(\frac{y}{C}\right) =& \int_{C}^y \int_{\mathbb{R}} \frac{\psi(\lambda)}{y'}\, d\lambda\, dy'.  \label{psi ident 1}
\end{align}
The identity (\ref{psi ident 1}) is trivial; its content is $\int \psi=1$.  Integrating (\ref{cRho1}) from $y$ to $C$ we get
\begin{align}
& \int_{a_-}^{a_+} \frac{\irho}{\lambda-y}\, d\lambda - \int_{-b}^b \frac{\irho}{\lambda-C}\, d\lambda \nonumber \\
& \qquad \sim  \log(y-a_+)  - \log(C-a_+)  + \int_{y}^C \int_{\mathbb{R}} \frac{\rho^{(1)}_N(\lambda)}{y'-\lambda}\, d\lambda\, dy' \nonumber \\
& \qquad =  \log\left(1-a_+/y\right)  + \log\left(\frac{y}{C}\right) + O(C^{-1}) + \int_{y}^C \int_{\mathbb{R}} \frac{\rho^{(1)}_N (\lambda)}{y'-\lambda }\, d\lambda\, dy' \nonumber \\
& \qquad  = \log\left(1-a_+/y\right)  +O(C^{-1}) \nonumber \\
& \qquad\qquad + \int_{y}^C \int_{\mathbb{R}} \left[\psi(\lambda) \left(\frac{1}{y'-\lambda}-\frac{1}{y'}\right) +N^{-2}W^{(-1)}_1(y') +\ldots\right] \, d\lambda \, dy' 
\end{align}
Since $W_1(y) - NW^{(1)}_1(y)$ is $O(y^{-2})$ as $y\rightarrow\infty$, we take the limit as $C\rightarrow\infty$.
\begin{align}
\int_{a_-}^{a_+} \frac{\irho}{\lambda-y}\, d\lambda \sim & \log\left(1-a_+/y\right)  
 + \int_{y}^\infty \left[ \int_{\mathbb{R}}  \frac{\lambda \psi(\lambda)/y'}{y'-\lambda}\, d\lambda +N^{-2}W^{(-1)}_1(y') +\ldots\right]  \, dy'  \label{W1expr1}
\end{align}
This proves formula (\ref{44eq}).

\subsection{Extraction of coefficients to prove formula (\ref{e1eq1})}

It is shown in \cite{EMcL03} that if $f(\lambda)$ is continuous and grows not worse than a polynomial at $\infty$, then the integral $\int f(\lambda) \rho^{(1)}_n(\lambda)\, d\lambda$ has an asymptotic expansion in powers on $N^{-2}$.  
Taking $f(\lambda)=V(\lambda)$ we see that the right hand side of the following equation has an $N\rightarrow\infty$ expansion.  
\begin{align}
\partial_x \log Z_N(x,t) =& \frac{N^2}{x} \int_{\mathbb{R}} V(\lambda) \rho^{(1)}_{N}(\lambda)\, d\lambda
\end{align}
Furthermore, it is shown in \cite{EMcL03} that the topological expansion can be differentiated term by term with respect to parameters; thus
\begin{align}
\partial_x \log \frac{Z_N(x,t)}{Z_N(x,0)} \sim & N^{2} \partial_x e_0 + \partial_x e_1 + N^{-2} \partial_x e_2 + \ldots \label{diff top}
\end{align}
Thus for some function $C_g(x)$ not depending on $t$ we have
\begin{align}
\partial_x e_g =&  C_g(x) + [N^{2-2g}]\partial_x \log Z_N.
\end{align}
This completes the proof of (\ref{e1eq1}).

%  %  %  %  %  %  %  %  %  %  %  %  %  %  %  %  %  %  %  %  %  %  %  %  %  %  %  %  
 %  %  %  %  %  %  %  %  %  %  %  %  %  %  %  %  %  %  %  %  %  %  %  %  %  %  %  %  
  %  %  %  %  %  %  %  %  %  %  %  %  %  %  %  %  %  %  %  %  %  %  %  %  %  %  %  %  
%  %  %  %  %  %  %  %  %  %  %  %  %  %  %  %  %  %  %  %  %  %  %  %  %  %  %  %  
 %  %  %  %  %  %  %  %  %  %  %  %  %  %  %  %  %  %  %  %  %  %  %  %  %  %  %  %  
  %  %  %  %  %  %  %  %  %  %  %  %  %  %  %  %  %  %  %  %  %  %  %  %  %  %  %  %  
%  %  %  %  %  %  %  %  %  %  %  %  %  %  %  %  %  %  %  %  %  %  %  %  %  %  %  %  
 %  %  %  %  %  %  %  %  %  %  %  %  %  %  %  %  %  %  %  %  %  %  %  %  %  %  %  %  
  %  %  %  %  %  %  %  %  %  %  %  %  %  %  %  %  %  %  %  %  %  %  %  %  %  %  %  %  
%  %  %  %  %  %  %  %  %  %  %  %  %  %  %  %  %  %  %  %  %  %  %  %  %  %  %  %  
 %  %  %  %  %  %  %  %  %  %  %  %  %  %  %  %  %  %  %  %  %  %  %  %  %  %  %  %  
  %  %  %  %  %  %  %  %  %  %  %  %  %  %  %  %  %  %  %  %  %  %  %  %  %  %  %  %  

\section{Calculation of $W_1^{(-1)}$ from loop equations}

We use the method loop equations to compute asymptotics of random matrix correlators.  This method is due to Eynard \cite{EynardLoopEqs}.  However, our calculations follow the presentation of this method in \cite{GuionnetBeta}.  The random matrix correlators are defined by
\begin{align}
W_{k,N}(y_1,\ldots,y_k) =& \partial_{\epsilon_1=0}\ldots \partial_{\epsilon_k =0} \log \int \exp\left(\sum_{i=1}^k \sum_{l=1}^N \frac{\epsilon_i}{y_i-\lambda_l}\right) e^{-N \sum_i V(\lambda_i)} \Delta^{2}(\lambda) \, d^N\lambda
\end{align}
The correlator $W_{k,N}$ is a generating function for maps on an oriented surface with $k$ punctures \cite{EynardLoopEqs}, but we will not use this fact here.  $W_{1,N}(y)$ is related to the Cauchy transform of the one point correlation function $\rho^{(1)}_N$ by
\begin{align}
W_{1,N}(y)=&
 N \int \frac{\rho^{(1)}_N(\lambda)}{y-\lambda}\, d\lambda .
%= -2\pi i N\, C\{\rho^{(1)}_N\}(y).
\end{align}

The correlators have asymptotic expansions in powers of $N$.  We define symbols $W^{(r)}_k$ for the coefficients of these expansions.
\begin{align}
W_{k,N}(y_1,\ldots,y_k) \sim & \sum_{g\geq 0} N^{2-k-2g} W_{r}^{(2-k-2g)}(y_1,\ldots,y_k) & \text{as }N\rightarrow\infty
\end{align}
The correlators are related by a sequence of ``loop equations''.  We use the loop equations essentially as they are stated in theorems 3.2, and 3.3 of \cite{GuionnetBeta}; the rank 1 and 2 loop equations are
\begin{align}
0\sim & W_{2}(y,y) + W_{1}(y)^2 - \frac{N}{2\pi i} \oint_C \frac{V'(\xi)}{y-\xi} W_{1}(\xi)\, d\xi \label{loop1} \\
0\sim & W_{3}(y,y,z) + 2 W_{1}(y) W_{2}(y,z) -\frac{N}{2\pi i}\oint_C  \frac{V'(\xi)}{y-\xi} W_{2}(\xi,z)\, d\xi + \partial_z \left(\frac{W_{1}(y)-W_{1}(z)}{y-z}\right) \label{loop2}
\end{align}
The loop equations in \cite{GuionnetBeta} are exact equalities, but we have neglected a pair of terms from each equation involving derivatives with respect to $a_\pm$.  These terms are small beyond all asymptotic orders, thus by neglecting them our loop equations become only equalities of asymptotic expansions.

We will extract terms at orders $N^1$ and $N^0$ from equations (\ref{loop2}) and (\ref{loop1}) respectively.  This will give equations from which we will determine $W_{2}^{(0)}$ and $W_{1}^{(1)}$.  The solution for these quantities is facilitated by an integral operator $K$.  It is proven in section 4.2 of \cite{GuionnetBeta} that the operator $K$ can be inverted using the following formula.
\begin{align}
Kf(y) =& 2 W^{(1)}_1(y)f(y) - \frac{1}{2\pi i} \oint_C \frac{V'(\xi)}{y-\xi}f(\xi)\, d\xi \\
K^{-1}g(y) =& \frac{-x}{R(y)} \frac{1}{2\pi i} \oint_C \frac{1}{(y-\xi)h(\xi)}g(\xi)\, d\xi \label{K inverse}
\end{align}
These formulas assume that $f(y)=O(y^{-2})$ as $y\rightarrow\infty$.  The contour $C$ winds once around $[\alpha_-,\alpha_+]$ and the point $y$ is exterior to the contour $C$.    Actually, we have slightly modified the formula from \cite{GuionnetBeta} by including the prefactor $x$, this arises because of a difference in notation.

\subsection{Calculation of $W_{2}^{(0)}$ from rank 2 loop equation}

We define a function $R(\lambda)$ analytic expect for a branch cut on $[\alpha_-,\alpha_+]$.  With this notation, we have the following expression for $\widetilde{\psi}$.
\begin{align}
R(\lambda)=& \sqrt{\lambda-\alpha_-}\sqrt{\lambda-\alpha_+}\\
\widetilde{\psi}(\lambda) =& \frac{1}{2\pi i x} R(\lambda) h(\lambda)
\end{align}
Extracting terms at order $N$ from the rank 2 loop equation, we find that
\begin{align}
0=& K_y W_{2}^{(0)}(y,z) + \partial_z \left(\frac{W_{1}^{(1)}(y)- W_{1}^{(1)}(z)}{y-z}\right).  \label{loop 2 order 1}
\end{align}
The subscript ``$K_y$'' means that the operator $K$ acts on the first argument of $W_{2}^{(0)}$, with $z$ behaving as a parameter.  The difference quotient in (\ref{loop 2 order 1}) can be expressed in terms of the equilibrium measure by
\begin{align}
\frac{W_{1}^{(1)}(y)- W_{1}^{(1)}(z)}{y-z} =&  \int_{\alpha_-}^{\alpha_+} \left( \frac{1}{y-\zeta } - \frac{1}{ z-\zeta }\right) \frac{\psi(\zeta)}{y-z}\, d\zeta \nonumber \\
=& - \int_{\alpha_-}^{\alpha_+} \frac{\psi(\zeta)}{(\zeta-y)(\zeta-z)}\, d\zeta.
\end{align}
By inverting the operator $K$ we solve for $W_{2}^{(0)}(y,z)$.  The points $y,z$ are outside the contour $C$.
\begin{align}
W_{2}^{(0)}(y,z) =& \frac{-x}{R(y)} \frac{1}{2\pi i} \oint_C \frac{1}{(y-\xi)h(\xi)} \partial_z \left( \int_{\alpha_-}^{\alpha_+} \frac{\psi(\zeta)\, d\zeta}{(\zeta-\xi)(\zeta-z)}\right)\, d\,\xi \nonumber \\
=& \frac{-x}{R(y)} \partial_z \int_{\alpha_-}^{\alpha_+} \frac{-\psi(\zeta)}{(y-\zeta)h(\zeta)(\zeta-z)}\, d\zeta \nonumber \\
=& \frac{-x}{R(y)} \partial_z \frac{1}{2} \oint_C \frac{\widetilde{\psi}(\zeta)}{(y-\zeta)h(\zeta)(\zeta-z)}\, d\zeta \nonumber \\
=& \frac{-1}{4\pi i R(y)} \partial_z \oint_C \frac{ R(\zeta)}{(y-\zeta)(\zeta-z)}\, d\zeta
\end{align}
We now move the contour $C$ past $y,z$, picking up residues of the integrand at these points.  The contour can then be shrunk to a small loop around $\infty$.  The residue at $\infty$ is constant with respect to $z$, so is killed by $\partial_z$.
\begin{align}
\ldots =& \frac{-1}{2 R(y)} \partial_z \left( \frac{R(z)-R(y)}{y-z} + \frac{1}{2\pi i}\oint_{\infty}\frac{ R(\zeta)}{(y-\zeta)(\zeta-z)}\, d\zeta\right) \nonumber \\
=& \frac{-1}{2 R(y)} \partial_z  \left( \frac{ R(z)- R(y)}{z-y}\right)
\end{align}
For the purpose of calculating $e_1$ we only need values of $W_{2}^{(0)}$ along the diagonal $y=z$.  
%Commuting the limit as $z\rightarrow y$ and using l'Hopital's rule, we obtain
\begin{align}
W_{2}^{(0)}(y,y) =& \frac{-1}{2 R(y)}  \partial_{z=y}  \frac{R(z)-R(y)}{z-y} \nonumber \\
%=& \frac{-1}{2 R(y)}   R''(y) \nonumber \\
=& \frac{-1/16}{(y-\alpha_-)^2} + \frac{1/8}{(y+\alpha_+)(y-\alpha_-)} + \frac{-1/16}{(y-\alpha_+)^2}
\end{align}

\subsection{Calculation of $W_{1}^{(-1)}$ from rank 1 loop equation}

Extracting terms at order $N^{0}$ from the rank 1 loop equation gives
\begin{align}
0=& W^{(0)}_2(y,y) +K W_{1}^{(-1)}(y) 
\end{align}
We solve for $W_{1}^{(-1)}(y)$ by inverting the operator $K$.
\begin{align}
W_{1}^{(-1)}(y) =&\frac{x}{R(y)} \frac{1}{2\pi i} \oint_C \frac{1}{(y-\xi)h(\xi)} W^{(0)}_2(\xi,\xi)\, d\xi \nonumber \\
=& \frac{x}{R(y)} \frac{1}{2\pi i} \oint_C \frac{1}{(y-\xi)h(\xi)} \left( \frac{-1/16}{(\xi-\alpha_-)^2} + \frac{1/8}{(\xi -\alpha_-)(\xi -\alpha_+)} + \frac{-1/16}{(\xi -\alpha_+)^2} \right) \, d\xi \nonumber \\
=& \frac{x}{R(y)}\bigg (\frac{h(\alpha_-)-2\sqrt{z}h'(\alpha_-)}{-32\sqrt{z} h(\alpha_-)^2(y-\alpha_-)}+\frac{1}{-16h(\alpha_-)(y-\alpha_-)^2}  \nonumber \\
& \qquad \qquad \frac{h(\alpha_+)+2\sqrt{z}h'(\alpha_+)}{32\sqrt{z} h(\alpha_+)^2(y-\alpha_+)}+\frac{1}{-16h(\alpha_+)(y-\alpha_+)^2} \bigg) \label{W 6 1}
\end{align}
To calculate $e_1$ we will need an indefinite integral of $W_{1}^{(-1)}$.  
This we compute using the following integral formulas, which can be proved by straitforward calculation.
\begin{align}
\int_{y}^\infty \frac{dy}{(y-\alpha_-)R(y)} =& \frac{1}{2\sqrt{z}}\left(1-\frac{(y-\alpha_+)}{R(y)}\right) \label{int 6 1}\\
\int_{y}^\infty \frac{dy}{(y-\alpha_-)^2 R(y)} =& \frac{1}{12z}\left(1-\frac{(y-\alpha_+)}{R(y)}-\frac{2\sqrt{z}(y-\alpha_+)}{(y-\alpha_-)R(y)}\right) \label{int 6 2}
\end{align}
The integrals $\int_{y}^\infty [(y-\alpha_+)R(y)]^{-1}\, dy$ and $\int_{y}^\infty [(y-\alpha_+)^2 R(y)]^{-1}\, dy$ are given by substituting $\alpha_{\pm}\mapsto\alpha_{\mp},\sqrt{z}\mapsto -\sqrt{z}$ in the above formulas. 
Using  (\ref{int 6 1}) and (\ref{int 6 2}) to integrate in (\ref{W 6 1}), we find that
\begin{align}
\frac{1}{x}\int_{y}^{\infty} W_{1}^{(-1)} =& \frac{2h(\alpha_-) -3\sqrt{z}h'(\alpha_-)}{-96z h(\alpha_-)^2}\left(1-\frac{y-\alpha_+}{R(y)}\right) \nonumber \\
&\qquad +\frac{1}{96\sqrt{z}h(\alpha_-)}\frac{y-\alpha_+}{(y-\alpha_-)R(y)}
+ \left(\begin{array}{>{\displaystyle}c}
\sqrt{z}\leftrightarrow -\sqrt{z} \\  
\alpha_{\pm} \leftrightarrow \alpha_{\mp}
\end{array}\right).
\end{align}
The term at the end with $\leftrightarrow$ arrows stands for all the preceeding terms again, but with the replacements $\sqrt{z}\mapsto -\sqrt{z}$ and $\alpha_\pm \mapsto \alpha_\mp$.

\section{Calculation of $e_1$ from $W_{1}^{(-1)}$}

Our calculations so far have shown that
\begin{align}
\partial_x e_1 =& C_1(x) - \frac{1}{x}\oint_C \widetilde{\psi}(y) \int_{y}^\infty W_{1}^{(-1)}(y')\, dy' \, dy \nonumber  \\
=& -\oint_{C} \frac{R(y)h(y)}{2\pi i\, x} \bigg( \frac{2h(\alpha_-) -3\sqrt{z}h'(\alpha_-)}{-96z h(\alpha_-)^2}\left(1-\frac{y-\alpha_+}{R(y)}\right)  \nonumber \\
& \qquad \qquad  +\frac{1}{96\sqrt{z}h(\alpha_-)}\frac{y-\alpha_+}{(y-\alpha_-)R(y)} +(\leftrightarrow )\bigg)\, dy
\end{align}
%($\widetilde{\psi}$ contributed a $1/x$ going to the second line.)
Our notation $X+ (\leftrightarrow)$ means the expression $X$ plus the result of replacing $\sqrt{z}\mapsto -\sqrt{z}, \alpha_{\pm}\mapsto \alpha_\mp$ in the expression $X$.  Here $X$ is understood to be all terms within in the same parenthesis as $(\leftrightarrow)$.
The function $W_{1}^{(-1)}$ can be separated into a contant part, and a $y$-dependent part.  Since the $\psi(\lambda)\, d\lambda$ is a probability measure, $\oint_C \widetilde{\psi}=-2$.  Thus the constant part of $W_{1}^{(-1)}$ contributes 
\begin{align}
&  -(-2)\left[\frac{2h(\alpha_-) -3\sqrt{z}h'(\alpha_-)}{-96z h(\alpha_-)^2} + \left( \leftrightarrow \right)\right] \nonumber \\
& \qquad = \frac{2h(\alpha_-) -3\sqrt{z}h'(\alpha_-)}{-48z h(\alpha_-)^2} + \frac{2h(\alpha_+) +3\sqrt{z}h'(\alpha_+)}{-48z h(\alpha_+)^2}
 \end{align}
 to $\partial_x e_1$.  Now we compute the contour integral of the $y$-dependent part.
 \begin{align}
& \oint_{C} \frac{R(y)h(y)}{2\pi i\, x}\bigg( \frac{2h(\alpha_-) -3\sqrt{z}h'(\alpha_-)}{96z h(\alpha_-)^2}\frac{(y-\alpha_+)}{-R(y)}-\frac{1}{96\sqrt{z}h(\alpha_-)}\frac{y-\alpha_+}{(y-\alpha_-)R(y)} +(\leftrightarrow )\bigg)\, dy \nonumber \\
& \qquad =\frac{1}{2\pi i\,x} \oint_C \text{holomorphic} + \left[\frac{-1}{96\sqrt{z}h(\alpha_-)}\frac{(y-\alpha_+)h(y)}{(y-\alpha_-)} +(\leftrightarrow )\right] \, dy \nonumber  \\
& \qquad =\frac{-1}{96x\sqrt{z}h(\alpha_-)} \frac{(\alpha_- - \alpha_+)h(\alpha_-)}{-1} + (\leftrightarrow) \nonumber \\
%& \qquad = \frac{\alpha_- - \alpha_+}{48 \sqrt{z}} \\
& \qquad = -\frac{1}{12x}
 \end{align}
Thus we obtain the following formula for $e_1$.  When $t=0$, the left hand side of (\ref{e1 1}) vanishes, and the equilibrium measure reduces to Wigner's semicircle; that is, at $t=0$ we have $z=1$ and $h(\lambda)=1$.  It follows that $C_1(x)$ and $-\frac{1}{12x}$ must cancel.
\begin{align}
\partial_x e_1 =& C_1(x) + \frac{2h(\alpha_-) -3\sqrt{z}h'(\alpha_-)}{-48z h(\alpha_-)^2} + \frac{2h(\alpha_+) +3\sqrt{z}h'(\alpha_+)}{-48z h(\alpha_+)^2} -\frac{1}{12x}  \nonumber \\
=& \frac{2h(\alpha_-) -3\sqrt{z}h'(\alpha_-)}{-48z h(\alpha_-)^2} + \frac{2h(\alpha_+) +3\sqrt{z}h'(\alpha_+)}{-48z h(\alpha_+)^2} \label{e1 1}
\end{align}
We now reformulate the above expression in terms of $u$ and $z$ using our valence independent formula for the equilibrium measure.  
\begin{align}
h(\alpha_+)=& \frac{1}{\sqrt{z}u_x +z_x} \label{h a 0} \\
h'(\alpha_+)=& \frac{\sqrt{z}u_{x}^2 + 3u_x z_x + 4z u_{xx} + 4\sqrt{z}z_{xx}}{-6(\sqrt{z}u_{x} +z_{x})^3}. \label{h a 1}
\end{align}
Note that, as observed in the introduction, formulas for $h(\alpha_-),h'(\alpha_-)$ are obtained by making substitutions $\sqrt{z}\mapsto -\sqrt{z}, \alpha_\pm \mapsto \alpha_{\mp}$ in (\ref{h a 0}),(\ref{h a 1}).
Substituting for the quantities $h(\alpha_\pm), h'(\alpha_\pm)$ in (\ref{e1 1}), we find that
\begin{align}
\partial_x e_1 =& \frac{-z_x}{12z } + \frac{z_x u_{x}^2 +2zu_x u_{xx} -2z_x z_{xx}}{24(zu_{x}^2 -z_{x}^2)}
\end{align}
Integrating both sides with respect to $x$ yields
\begin{align}
e_1 =& -\frac{1}{12}\log z + \frac{1}{24} \log\left(z u_{x}^2 -z_{x}^2\right). \label{val indep e1}
\end{align}
One might expect the right and left hand sides to differ by a constant of integration, but we now show that is not the case: since $e_1$ is a generating function for maps on the torus, and a map must have at least one vertex, it follows that $e_1(t=0)=0$.  Also, at $t=0$ our matrix ensemble reduces to the GUE (except that the potential is scaled by $1/x$), so we obtain the evaluations $z=x, z_x=1, u=0$ and $u_x=1$ from Wigner's semicircle.  One easily checks that both sides match when $t=0$ and therefore do not differ by a constant.

\section{Special case of $e_1$ formula for regular maps}
The formula obtained above for $e_1$ is valid for arbitrary an polynomial potential.  Let us consider the special case of a regular map, that is, where each vertex has the same valence $j$.  This corresponds to a potential of the form
\begin{align}
V(\lambda) =& \frac{1}{x} \left( \frac{1}{2}\lambda^2 + t \lambda^j\right). \label{special V}
\end{align}
Three different proofs of the following equations will appear in P. Waters dissertation.
\begin{align}		
u + jt u_t =& 2\partial_x(uz) 	\label{uToda a}	\\
2z +jt  z_t =& z \partial_x (u^2) +\partial_x (z^2) \label{uToda b}
\end{align}
These equations can be proved by taking a combination of string and Toda equations; or by introducing an ``edge counting variable'' in the matrix model and constructing a Toda-like equation in this variable; or by expressing a particular Virasoro constraint for the matrix model in terms of recurrence coefficients for orthogonal polynomials.  

The following scaling relations are well known \cite{EP11}.
\begin{align}		
(j-2)t  z_t + 2z	=&	2x  z_x	\label{scaling 1}\\
(j-2) t  u_t + u =& 2x u_x		\label{scaling 2}
\end{align}
Now use (\ref{scaling 1}) and (\ref{scaling 2}) to eliminate $t$ derivatives in (\ref{uToda a}) and (\ref{uToda b}).  Solving the resulting system for $u_x$ and $z_x$ gives following derivative reducing rules.
\begin{align}
z_{x}=& \frac{2z(jx-(j-2)z)+(j-2)u^2z}{(jx-(j-2)z)^2-(j-2)^2u^2z} \label{z deriv} \\
u_x=& \frac{jxu+(j-2)uz}{(jx-(j-2)z)^2-(j-2)^2u^2z} \label{u deriv}
\end{align}
Use (\ref{z deriv}) and (\ref{u deriv}) to replace the quantities $u_x$ and $z_x$ in formula (\ref{val indep e1}), and set $x=1$.  This gives the following formula for $e_1$, valid for potentials of the form $V(\lambda) = \frac{1}{2}\lambda^2 + t \lambda^j$.
\begin{align}
e_1 =& \frac{1}{24} \log \frac{4-u^2/z}{(j-(j-2)z)^2-(j-2)^2u^2z}.
\end{align}
For an even valence ($u=0$), this reduces to the formula (\ref{EMcLP e1}) of Ercolani, McLaughlin and Pierce.

\end{document}